\documentclass[12pt,twoside]{amsart}
\usepackage{amsmath, amsthm, amscd, amsfonts, amssymb, graphicx}
\usepackage{enumerate}
\usepackage[colorlinks=true,
linkcolor=blue,
urlcolor=cyan,
citecolor=red]{hyperref}
\usepackage{mathrsfs}
\newtheorem{theorem}{Theorem}[section]
\newtheorem{lemma}{Lemma}[section]
\newtheorem{remark}{Remark}[section]
\newtheorem{definition}{Definition}[section]
\newtheorem{corollary}{Corollary}[section]

\newtheorem{proposition}{Proposition}[section]
\numberwithin{equation}{section}

\begin{document}
	
\title{On the matrix Cauchy-Schwarz  inequality}
\author{Mohammad Sababheh$^{*}$, Cristian Conde, and Hamid Reza Moradi}

\begin{abstract}
The main goal of this work is to present new matrix inequalities of the Cauchy-Schwarz type. In particular, we investigate the so-called Lieb functions, whose definition came as an umbrella of Cauchy-Schwarz-like inequalities, then we consider the mixed Cauchy-Schwarz inequality. This latter inequality has been influential in obtaining several other matrix inequalities, including numerical radius and norm results. 
Among many other results, we show that
\[\left\| T \right\|\le \frac{1}{4}\left( \left\| \left| T \right|+\left| {{T}^{*}} \right|+2\mathfrak RT \right\|+\left\| \left| T \right|+\left| {{T}^{*}} \right|-2\mathfrak RT \right\| \right),\]
where $\mathfrak RT$ is the real part of $T$.
\end{abstract}
\maketitle
\pagestyle{myheadings}
\markboth{\centerline {}}
{\centerline {}}
\bigskip
\bigskip

\section{Introduction}
Let $\mathcal{M}_n$ be the algebra of all $n\times n$ complex matrices, with zero element $O$.  The matrix Cauchy-Schwarz (for brevity, we write CS) inequality states that if $A,B,X\in\mathcal{M}_n$ and if $|||\cdot|||$ is any unitarily invariant norm on $\mathcal{M}_n$, then \cite[Theorem IX.5.1]{1}
\begin{equation}\label{eq_matrix_CS}
|||A^*XB|||^2\leq |||AA^*X|||\;|||XBB^*|||,
\end{equation}
where $A^*$ denotes the conjugate of the matrix $A$. Notice that if we let $A={\text{diag}}[x_i]$ and $B={\text{diag}}[y_i]$, then the scalar CS inequality follows from \eqref{eq_matrix_CS}, by letting $X=I$; the identity matrix.

This article aims to discuss further the matrix CS inequality and the mixed CS inequality. This will be done in two sections; where in the first section, we deal with a class of functions, namely Lieb functions, that was initially defined as a class of functions satisfying CS-type inequalities, then we present more elaborated mixed CS inequalities.

Recall that if $A,B\in\mathcal{M}_n$ are positive definite (denoted as $A,B>O$) the geometric mean of $A,B$ is given by $A\sharp B=A^{\frac{1}{2}}\left(A^{-\frac{1}{2}}BA^{-\frac{1}{2}}\right)^{\frac{1}{2}}A^{\frac{1}{2}}.$ The geometric mean of matrices has received considerable attention in the literature due to its significance in understanding the geometry of $\mathcal{M}_n$. We refer the reader to \cite{ando_hiai,drag,drury,Furuta,hms,lin,liu} as a sample of references dealing with this concept. 

Among many interesting results about Lieb functions, we show that these functions satisfy the Ando-type inequality 
$$f(A\sharp B)\leq f(A)\sharp f(B).$$ 
On the other hand, as a conclusion to the discussion of the mixed CS inequality, we show that for $T\in\mathcal{M}_n$,
\[\left\| T \right\|\le \frac{1}{2}\left( \left\| \frac{{{f}^{2}}\left( \left| T \right| \right)+{{g}^{2}}\left( \left| {{T}^{*}} \right| \right)}{2}+\mathfrak RT \right\|+\left\| \frac{{{f}^{2}}\left( \left| T \right| \right)+{{g}^{2}}\left( \left| {{T}^{*}} \right| \right)}{2}-\mathfrak RT \right\| \right),\]
where $f,g:[0,\infty)\to [0,\infty)$ satisfy $f(t)g(t)=t$.  Here, $|A|=(A^*A)^{\frac{1}{2}}$ and $\mathfrak{R}T$ is the real part of $T$, defined by $\mathfrak{R}T=\frac{T+T^*}{2}.$
This inequality is significant since the following inequality is not valid: $||T||\leq\frac{1}{2}||f^2(|T|)+g^2(|T^*|)||$; see Corollary \ref{nee1} for better comprehension of this result. Many other results, including sums, products, and convex combinations, will be shown too.

\section{Lieb functions}
In \cite{seiler}, the following inequality among determinants was shown as a possible determinant CS inequality
\begin{equation}\label{eq_det}
|\det(I+A+B)|\leq \det(I+|A|)\det(I+|B|),
\end{equation}
where $A,B\in\mathcal{M}_n$.

One year later, Lieb \cite{3} extended this inequality to a broader class of matrix functions, denoted by $\mathcal{L},$ and defined as follows. 
\begin{definition}
Let $f:\mathcal{M}_n\to\mathbb{C}$ be a complex valued function defined on $\mathcal{M}_n$. We say that $f\in\mathcal{L}$ if it satisfies the following two properties for $A,B\in\mathcal{M}_n$:
\begin{itemize}
\item[(i)] $f\left( B \right)\ge f\left( A \right)\geq 0$ if $B\ge A\ge O$;
\item[(ii)]  ${{\left| f\left( {{A}^{*}}B \right) \right|}^{2}}\le f\left( {{A}^{*}}A \right)f\left( {{B}^{*}}B \right)$ for all $A,B$.
\end{itemize}
\end{definition}

In this context, we write $A\geq O$ if $\left<Ax,x\right>\geq 0$ for all $x\in\mathbb{C}^n$, while we write $A\geq B$ if $A-B\geq O.$ We distinguish the zero matrix $O$ from the scalar 0 due to the differences between the ranges of the functions we deal with.
Examples of functions in $\mathcal L$ are the determinant, permanent, spectral radius, any elementary symmetric function of the eigenvalues, and any unitarily invariant norm; see \cite{1,3,5}.\\
The second inequality in the above definition is CS-like inequality. It is also implicitly understood, from the first property, that when $A\geq O$ then $f(A)\geq 0.$

Lieb functions were discussed in detail in the literature, and some exciting properties were shown for this class. We refer the reader, in particular, to \cite{3} for the original definition and the basic properties, while we refer the reader to \cite[Section IX.5]{1} for other properties of the class $\mathcal{L}$, and its connection to matrix inequalities. We also refer the reader to \cite{akhrass} for further discussion of Lieb functions and sectorial matrices.

The significance of discussing the class $\mathcal{L}$ is the way it retrieves properties of this class examples, such as the determinant and unitarily invariant norms, for example.

This section investigates  further properties of this class and shows many interesting inequalities that extend some known inequalities for particular examples of such functions. For example, we show that if $f\in\mathcal{L}$ and $T\in\mathcal{M}_n$, then
$$|f(T)|^2\leq f(|T|)f(|T^*|)$$ as a CS-type inequality; then weighted form of this inequality will be presented. Many other interesting inequalities, including relations among $f(A+B)$ and $f(|A|+|B|)f(|A^*|+|B^*|)$, will be presented; however, we will discuss such relations in a more general scenario.

To achieve our target in this section, we will need some lemmas, as follows. 
\begin{lemma}\label{3}
\cite[Theorem IX.5.10]{1} A continuous function $f:\mathcal{M}_n\to\mathbb{C}$ is in the class $\mathcal L$ if and only if it  satisfies the following two conditions:
\begin{itemize}
\item[(i)] $f\left( A \right)\ge 0$ for all $A\ge O$;
\item[(ii)]  ${{\left| f\left( C \right) \right|}^{2}}\le f\left( A \right)f\left( B \right)$ for all $A,B,C$ such that $\left[ \begin{matrix}
   A & {{C}^{*}}  \\
   C & B  \\
\end{matrix} \right]\geq O.$
 \end{itemize}
\end{lemma}

\begin{lemma}\label{1}
\cite[Lemma 2]{4} Let $A,B,C\in\mathcal{M}_n$ be such that $A,B\geq O$  and $BC=CA$. If $\left[ \begin{matrix}
   A & {{C}^{*}}  \\
   C & B  \\
\end{matrix} \right]\geq O$, then $\left[ \begin{matrix}
   {{g}^{2}}\left( A \right) & {{C}^{*}}  \\
   C & {{h}^{2}}\left( B \right)  \\
\end{matrix} \right]\geq O$
for any pair of continuous non-negative functions $g,h$ on $\left[ 0,\infty  \right)$ satisfying $g\left( t \right)h\left( t \right)=t$ for all $t\in \left[ 0,\infty  \right)$.
\end{lemma}
We should remark that when $f:[0,\infty)\to [0,\infty)$ is a given function and $A\geq O$, then $f(A)$ is defined via functional calculus by $$f(A)=U{\text{diag}}[f(\lambda_i)]U^*,$$ where $U$ is a unitary matrix such that $A=U{\text{diag}}[\lambda_i]U^*$; the spectral decomposition of $A$.

\begin{lemma}\label{2}
\cite[Lemma 3]{4} For any $T\in\mathcal{M}_n$, 
$$\left[ \begin{matrix}
   \left| T \right| & {{T}^{*}}  \\
   T & \left| {{T}^{*}} \right|  \\
\end{matrix} \right]\ge O.$$
\end{lemma}

Our first main result in this section reads as follows.
\begin{theorem}\label{thm_|T|}
Let $T\in {{\mathcal M}_{n}}$ and let $f\in \mathcal L$. If $g, h$ are non-negative continuous functions on $\left[ 0,\infty  \right)$ satisfying $g\left( t \right)h\left( t \right)=t$, ($t\ge 0$), then
\[{{\left| f\left( T \right) \right|}^{2}}\le f\left( {{g}^{2}}\left( \left| T \right| \right) \right)f\left( {{h}^{2}}\left( \left| {{T}^{*}} \right| \right) \right).\]
In particular,
$$|f(T)|^2\leq f(|T|)f(|T^*|).$$ 
\end{theorem}
\begin{proof}
We mimic some ideas from \cite{4}. Since $T{{\left| T \right|}^{2}}={{\left| {{T}^{*}} \right|}^{2}}T$, it follows that $T\left| T\right|=\left| {{T}^{*}} \right|T$. Hence,
 by Lemmas \ref{1} and \ref{2}, we have 
 \begin{equation}\label{9}
\left[ \begin{matrix}
     {{g}^{2}}\left( \left| T \right| \right) & {{T}^{*}}  \\
     T & {{h}^{2}}\left( \left| {{T}^{*}} \right| \right)  \\
  \end{matrix} \right]\ge O.
 \end{equation}
The first required inequality now follows from Lemma \ref{3}. For the second inequality, we let $g(t)=h(t)=\sqrt{t}$ in the first inequality. This completes the proof.
\end{proof}
Noting that unitarily invariant norms and determinants are Lieb functions, Theorem \ref{thm_|T|} implies 
$$|||T|||^2\leq |||\;|T|\;|||\;|||\;|T^*|\;|||\;{\text{and}}\;|\det T|^2\leq \det|T|\det|T^*|,\;T\in\mathcal{M}_n.$$ This follows from $g(t)=h(t)=\sqrt{t}.$ However, if $0\leq v\leq 1$, then $g(t)=t^{v}$ and $h(t)=t^{1-v}$ satisfy the assumptions of Theorem \ref{thm_|T|}. Applying these functions imply
$$|||T|||^2\leq |||\;|||\;|T|^{2v}|||\;|||\;|T^*|^{2(1-v)}|||\;{\text{and}}\;|\det T|^2\leq \det|T|^{2v}\det|T^*|^{2(1-v)}.$$

An exciting consequence of Theorem \ref{thm_|T|} is the following weighted inequality.
\begin{corollary}\label{4}
Let $T\in {{\mathcal M}_{n}}$ and let $f\in \mathcal L$. Then for any $0\le t\le 1$,
\[{{\left| f\left( T \right) \right|}^{2}}\le f\left( {{\left| {{T}^{*}} \right|}^{2\left( 1-t \right)}} \right)f\left( {{\left| T \right|}^{2t}} \right).\]
\end{corollary}

\begin{remark}
If $T\in {{\mathcal M}_{n}}$ is normal, we infer from Corollary \ref{4} that
\[\left| f\left( T \right) \right|\le f\left( \left| T \right| \right).\]
We also remark that the positivity of the matrix
$\left[ \begin{matrix}
   {{\left| T \right|}^{2}} & {{T}^{*}}  \\
   T & I  \\
\end{matrix} \right]$ (see, e.g., \cite[Proposition 1.7.7]{7}), implies
\[{{\left| f\left( T \right) \right|}^{2}}\le f\left( {{\left| T \right|}^{2}} \right)f\left( I \right);\] which also follows from Corollary \ref{4} by letting $t=1$.
\end{remark}

In studying matrix inequality, the quantity $f(A+B)$, for certain $f$, has received a renowned interest. For example, if $f:[0,\infty)\to[0,\infty)$ is concave and $A,B\geq O$, one has the inequality 
$$|||f(A+B)|||\leq |||f(A)+f(B)|||,$$ for any unitarily invariant norm $|||\cdot|||$ on $\mathcal{M}_n$; see \cite{bourin}. On the other hand, the quantity $f((1-v)A+vB)$ has significance in convex matrix inequalities. The following result presents possible inequalities for these quantities via Lieb functions.

\begin{theorem}
Let $A,B\in {{\mathcal M}_{n}}$ and let $f\in \mathcal L$. If $g, h$ are non-negative continuous functions on $\left[ 0,\infty  \right)$ satisfying $g\left( t \right)h\left( t \right)=t$, ($t\ge 0$), then
\begin{equation}\label{10}
{{\left| f\left( A+B \right) \right|}^{2}}\le f\left( {{g}^{2}}\left( \left| A \right| \right)+{{g}^{2}}\left( \left| B \right| \right) \right)f\left( {{h}^{2}}\left( \left| {{A}^{*}} \right| \right)+{{h}^{2}}\left( \left| {{B}^{*}} \right| \right) \right),
\end{equation}
and
\begin{equation}\label{11}
{{\left| f\left( \left( 1-v \right)A+vB \right) \right|}^{2}}\le f\left( \left( 1-v \right){{g}^{2}}\left( \left| A \right| \right)+v{{g}^{2}}\left( \left| B \right| \right) \right)f\left( \left( 1-v \right){{h}^{2}}\left( \left| {{A}^{*}} \right| \right)+v{{h}^{2}}\left( \left| {{B}^{*}} \right| \right) \right),
\end{equation}
for any $0\le v\le 1$.
\end{theorem}
\begin{proof}
By \eqref{9}, we have
	\[\left[ \begin{matrix}
   {{g}^{2}}\left( \left| A \right| \right) & {{A}^{*}}  \\
   A & {{h}^{2}}\left( \left| {{A}^{*}} \right| \right)  \\
\end{matrix} \right]\ge O,\;\left[ \begin{matrix}
   {{g}^{2}}\left( \left| B \right| \right) & {{B}^{*}}  \\
   B & {{h}^{2}}\left( \left| {{B}^{*}} \right| \right)  \\
\end{matrix} \right]\ge O,\]
and
\[\left[ \begin{matrix}
   \left( 1-v \right){{g}^{2}}\left( \left| A \right| \right) & \left( 1-v \right){{A}^{*}}  \\
   \left( 1-v \right)A & \left( 1-v \right){{h}^{2}}\left( \left| {{A}^{*}} \right| \right)  \\
\end{matrix} \right]\ge O,\;\left[ \begin{matrix}
   v{{g}^{2}}\left( \left| B \right| \right) & v{{B}^{*}}  \\
   vB & v{{h}^{2}}\left( \left| {{B}^{*}} \right| \right)  \\
\end{matrix} \right]\ge O.\]
Consequently,
\[\left[ \begin{matrix}
   {{g}^{2}}\left( \left| A \right| \right)+{{g}^{2}}\left( \left| B \right| \right) & {{\left( A+B \right)}^{*}}  \\
   A+B & {{h}^{2}}\left( \left| {{A}^{*}} \right| \right)+{{h}^{2}}\left( \left| {{B}^{*}} \right| \right)  \\
\end{matrix} \right]\ge O,\]
and
\[\left[ \begin{matrix}
   \left( 1-v \right){{g}^{2}}\left( \left| A \right| \right)+v{{g}^{2}}\left( \left| B \right| \right) & {{\left( \left( 1-v \right)A+vB \right)}^{*}}  \\
   \left( 1-v \right)A+vB & \left( 1-v \right){{h}^{2}}\left( \left| {{A}^{*}} \right| \right)+v{{h}^{2}}\left( \left| {{B}^{*}} \right| \right)  \\
\end{matrix} \right]\ge O.\]
Employing Lemma \ref{3} (ii) completes the proof.
\end{proof}

From \eqref{10}, one can infer the following known inequality
\begin{equation}\label{5}
|||A+B|||^2\leq |||\;|A|^{2v}+|B|^{2v}|||\;|||\;|A^*|^{2(1-v)}+|B^*|^{2(1-v)}|||,0\leq v\leq 1,
\end{equation}
for any $A,B\in\mathcal{M}_n$ and any unitarily invariant norm $|||\cdot|||$.

Let $\Phi :{{\mathcal M}_{n}}\to {{\mathcal M}_{m}}$ a 2-positive linear map, let $Z\in {{\mathcal M}_{n}}$ and let $p\in \left( -\infty ,\infty  \right)$. Then, from \cite[Corollary 2.7]{rs}, there exists a unitary $V\in {{\mathcal M}_{m}}$ such that
	\[\left| \Phi \left( Z \right) \right|\le \Phi \left( {{\left| Z \right|}^{1+p}} \right)\sharp V\Phi \left( {{\left| {{Z}^{*}} \right|}^{1-p}} \right){{V}^{*}}.\]
Letting $Z=A\oplus B$ and $\Phi \left( A\oplus B \right)=A+B$, we get a considerable improvement of \eqref{5}, and we also see that it is not necessary to confine to $0\le v \le 1$.

The following special case deserves mentioning, where relations between $|f(T)|$ and $f(|\mathfrak{R}T|+|\mathfrak{I}T|)$ are presented. Here, $\mathfrak{I}T$ refers to the  imaginary part of $T$, defined by $\mathfrak{I}T=\frac{T-T^*}{2i}.$

\begin{corollary}
Let $T\in {{\mathcal M}_{n}}$ and let $f\in \mathcal L$. If $g, h$ are non-negative continuous functions on $\left[ 0,\infty  \right)$ satisfying $g\left( t \right)h\left( t \right)=t$, ($t\ge 0$), then
\[{{\left| f\left( T \right) \right|}^{2}}\le f\left( {{g}^{2}}\left( \left| \mathfrak RT \right| \right)+{{g}^{2}}\left( \left| \mathfrak IT \right| \right) \right)f\left( {{h}^{2}}\left( \left| \mathfrak RT \right| \right)+{{h}^{2}}\left( \left| \mathfrak IT \right| \right) \right).\]
In particular, for any $0\le t\le 1$,
\[{{\left| f\left( T \right) \right|}^{2}}\le f\left( {{\left| \mathfrak RT \right|}^{2t}}+{{\left| \mathfrak IT \right|}^{2t}} \right)f\left( {{\left| \mathfrak RT \right|}^{2\left( 1-t \right)}}+{{\left| \mathfrak IT \right|}^{2\left( 1-t \right)}} \right).\]
\end{corollary}
\begin{proof}
If we replace $B$ by $\textup i B$, in  \eqref{10}, we get
\[{{\left| f\left( A+\textup iB \right) \right|}^{2}}\le f\left( {{g}^{2}}\left( \left| A \right| \right)+{{g}^{2}}\left( \left| B \right| \right) \right)f\left( {{h}^{2}}\left( \left| {{A}^{*}} \right| \right)+{{h}^{2}}\left( \left| {{B}^{*}} \right| \right) \right).\]
Now, if $T=A+\textup iB$ with $A=\frac{T+{{T}^{*}}}{2}=\mathfrak RT$ and $B=\frac{T-{{T}^{*}}}{2\textup i}=\mathfrak IT$, we get the desired result.
\end{proof}

As a direct consequence of inequality \eqref{11}, we have the following result; where $f(\mathfrak{R}T)$ and $f(\mathfrak{I}T)$ are compared as lower bounds of certain Lieb functions of $|T|$ and $|T^*|.$
\begin{corollary}
Let $T\in {{\mathcal M}_{n}}$ and let $f\in \mathcal L$. If $g, h$ are non-negative continuous functions on $\left[ 0,\infty  \right)$ satisfying $g\left( t \right)h\left( t \right)=t$, ($t\ge 0$), then
\[{{\left| f\left( \mathfrak RT \right) \right|}^{2}}\le f\left( \frac{{{g}^{2}}\left( \left| T \right| \right)+{{g}^{2}}\left( \left| {{T}^{*}} \right| \right)}{2} \right)f\left( \frac{{{h}^{2}}\left( \left| T \right| \right)+{{h}^{2}}\left( \left| {{T}^{*}} \right| \right)}{2} \right),\]
and
\[{{\left| f\left( \mathfrak IT \right) \right|}^{2}}\le f\left( \frac{{{g}^{2}}\left( \left| T \right| \right)+{{g}^{2}}\left( \left| {{T}^{*}} \right| \right)}{2} \right)f\left( \frac{{{h}^{2}}\left( \left| T \right| \right)+{{h}^{2}}\left( \left| {{T}^{*}} \right| \right)}{2} \right).\]
\end{corollary}

In the following result, we see the action of a Lieb function on the geometric mean of two positive definite matrices. 
Of particular interest, researchers investigated how the geometric mean behaves under certain functions. For example, if $\Phi:\mathcal{M}_n\to\mathcal{M}_n$ is a unital positive linear mapping, it is known that  $\Phi(A\sharp B)\leq \Phi(A)\sharp \Phi(B);$ \cite{az}. In the following, we see similar results for Lieb functions. We state here that this result is well-known in the literature, but we include its proof for the reader's convenience.
\begin{remark}\label{8}
Let $A,B\in {{\mathcal M}_{n}}$ be two positive definite matrices and let $f\in \mathcal L$. Then
\[{{f}^{2}}\left( A\sharp B \right)\le f\left( A \right)f\left( B \right).\]
In particular, $f(A\sharp B)\leq f(A)\sharp f(B).$ Indeed, we know that (see \cite[(4.13)]{6})
\begin{equation}\label{6}
\left[ \begin{matrix}
   A & A\sharp B  \\
   A\sharp B & B  \\
\end{matrix} \right]\ge O.
\end{equation}
Now, we deduce the desired result by Lemma \ref{3} (ii).

From the above discussion, one can say that the restriction of a Lieb function on the class of positive definite matrices is a geometrically convex function \cite{tdnd}. 
\end{remark}

As an interesting consequence of Remark \ref{8}, we have the following result about log-convexity of the function $t\mapsto f(A\sharp_tB).$
\begin{corollary}
Let $A,B\in {{\mathcal M}_{n}}$ be two positive definite matrices and let $f\in \mathcal L$. Then the function
	\[g\left( t \right)=f\left( A{{\sharp}_{t}}B \right)\]
is log-convex on $\left[ 0,1 \right]$.
\end{corollary}
\begin{proof}
We check the mid-log-convex condition. Let $0\le s,t\le 1$. Then 
	\[\begin{aligned}
   g\left( \frac{t+s}{2} \right)&=f\left( A{{\sharp}_{\frac{t+s}{2}}}B \right) \\ 
 & =f\left( \left( A{{\sharp}_{t}}B \right)\sharp \left( A{{\sharp}_{s}}B \right) \right) \\ 
 & \le f\left( A{{\sharp}_{t}}B \right)\sharp f\left( A{{\sharp}_{s}}B \right) \\ 
 & =f\left( t \right)\sharp f\left( s \right).  
\end{aligned}\]
This completes the proof.
\end{proof}

Now we present a convex generalization of condition (ii) in the definition of the class $\mathcal L$.
\begin{proposition}
Let $A,X,B\in {{\mathcal M}_{n}}$ and let $f\in \mathcal L$. Then for any $0\le t\le 1$,
\begin{equation}
\label{gencondii}|f\left( tB^*A+(1-t)A^*B \right)|^2\le f\left( tA^*A+(1-t)B^*B \right)f\left( (1-t)tA^*A+tB^*B \right).
\end{equation}
\end{proposition}\begin{proof}
Denote 
$$T=\left[\begin{matrix}
A^*  \\
B^*  \\
\end{matrix} \right] \left[\begin{matrix}
A & B  \\
\end{matrix} \right]=\left[\begin{matrix}
A^*A & A^*B \\
B^*A& B^*B  \\
\end{matrix} \right]\quad \textrm{and} \quad S=\left[\begin{matrix}
B*  \\
A^*  \\
\end{matrix} \right] \left[\begin{matrix}
B & A  \\
\end{matrix} \right]=\left[\begin{matrix}
B^*B & B^*A \\
A^*B& A^*A  \\
\end{matrix} \right].$$
Then, $T, S\geq O$, $tT+(1-t)S\geq O$ for $0\leq t\leq 1$.
As $f$ is 2-positive, we get
\[\left[ \begin{matrix}
f\left(tA^*A+(1-t)B^*B) \right) & f\left( t(A^*B+(1-t)B^*A) \right)  \\
f\left( tB^*A+(1-t)A^*B) \right) & f\left( tB^*B+(1-t)A^*A\right)  \\
\end{matrix} \right]\ge O.\]
Employing Lemma \ref{3} (ii), we obtain 
\[|f\left( tB^*A+(1-t)A^*B \right)|^2\le f\left( tA^*A+(1-t)B^*B \right)f\left( (1-t)tA^*A+tB^*B \right),\] 
which completes the proof.
\end{proof}
\begin{remark}
	For $t=0$ or $t=1$, inequality \eqref{gencondii}  reduces to condition (ii) in the definition of the Lieb functions. For $t=\frac 12$, on the other hand, this yields 
	\begin{equation*}
\left| f\left( \mathfrak R\left({{B}^{*}}A \right)\right) \right|=\left|f\left( \frac12B^*A+\frac 12A^*B \right)\right|\le f\left( \frac 12A^*A+\frac 12B^*B \right),
	\end{equation*}
	which is an interesting gathering inequality for Lieb functions. 
	\end{remark}

Another attractive property of Lieb functions together with the geometric mean is the following. Notice that this result extends the inequality in Remark \ref{8}.
\begin{proposition}
			Let $A_1, A_2, C, B_1, B_2\in {{\mathcal M}_{n}}$ be such that $A_i,B_i>O$ and $ \left[\begin{matrix}
				A_j & C^*  \\
				C & B_j  \\
			\end{matrix} \right]\ge O.$ If $f\in \mathcal L,$ then 
			\[|f(C)|^2\leq f\left( A_1\sharp A_2 \right)f\left( B_1\sharp B_2 \right).\]
		\end{proposition}
	\begin{proof}
		By \cite[Lemma 3.1]{8}, we have $\left[\begin{matrix}
			A_1\sharp A_2 & C^*  \\
			C &  B_1\sharp B_2  \\
		\end{matrix} \right]\ge O.$ Then  Lemma \ref{3} implies
		\[|f(C)|^2\leq f\left( A_1\sharp A_2 \right)f\left( B_1\sharp B_2 \right),\]
		which completes the proof.
		\end{proof}
	Another consequence of block characterizations can be stated as follows.
\begin{proposition}\label{offblock}
		Let $A, B,  C \in {{\mathcal M}_{n}}$ be such that $ \left[\begin{matrix}
	A & C^*  \\
	C & B  \\
	\end{matrix} \right]\ge O$ and $f\in \mathcal L.$ Then 
	\[|f(C+C^*)|^2\leq f^2\left( A+B\right).\]
	\end{proposition}
	\begin{proof}
		Since $\left[\begin{matrix}
			A & C^*  \\
			C & B  \\
		\end{matrix} \right]\ge O,$ then $\left[\begin{matrix}
			B & C  \\
			C^* & A  \\
		\end{matrix} \right]\ge O$. Now, we deduce  from Lemma \ref{3} (ii) and the positive matrix $\left[\begin{matrix}
			A+B & C+C^*  \\
			C+C^* & B+A  \\
		\end{matrix} \right]$ that $|f(C+C^*)|^2\leq f^2\left( A+B\right).$
	\end{proof}
	
	This entails the following result about a Lieb function of the real part of a matrix.
\begin{corollary}\label{offblockT}
 Let $T\in  {{\mathcal M}_{n}}$  and $f\in \mathcal L.$ Then
 \[|f(T+T^*)|^2\leq f^2(|T|^2+I),\]
 and 
 \[|f(T+T^*)|^2\leq f^2(|T|+|T^*|).\]
Equivalently,
 \[|f(2\mathfrak RT)|\leq \min\{f(|T|^2+I), f(|T|+|T^*|)\}.\]
	\end{corollary}
\begin{proof}
	This follows from Proposition \ref{offblock} via 
	\[\left[ \begin{matrix}
	\left| T \right| & {{T}^{*}}  \\
	T & \left| {{T}^{*}} \right|  \\
	\end{matrix} \right]\ge O\qquad \textrm {and}\qquad \left[ \begin{matrix}
	{{\left| T \right|}^{2}} & {{T}^{*}}  \\
	T & I  \\
	\end{matrix} \right]\geq O.\]
\end{proof}
\begin{remark}
	If $T\in {{\mathcal M}_{n}}$ is normal, we infer from Corollary \ref{offblockT} that
	\[\left| f(T+T^*) \right|^2\le f^2\left( 2\left| T \right| \right).\]
	In particular, if $T$ is self-adjoint, then 
		\[\left| f(2T) \right|^2\le f^2\left( 2\left| T \right| \right).\]
	\end{remark}

\section{Mixed Cauchy-Schwarz inequality}
In this section, we elaborate more on the mixed CS inequality.
\begin{lemma}\label{04}
\cite[p. 45]{r9} Let $T \geq O$  and let $x,y$  be any vectors. Then
\[\left| \left\langle Tx,y \right\rangle  \right|{^2}\le \left\langle Tx,x \right\rangle \left\langle Ty,y \right\rangle.\]
\end{lemma}
The following lemma will also be needed in our analysis.
\begin{lemma}\label{01}
(see \cite{04} and \cite[Lemma 1.1]{01}) For every block positive matrix in $\mathcal M_{n+m}$, the following decomposition holds
\[\left[ \begin{matrix}
   A & C  \\
   {{C}^{*}} & B  \\
\end{matrix} \right]=U\left[ \begin{matrix}
   A & O  \\
   O & O  \\
\end{matrix} \right]{{U}^{*}}+V\left[ \begin{matrix}
   O & O  \\
   O & B  \\
\end{matrix} \right]{{V}^{*}}\]
for some unitaries $U,V \in \mathcal M_{n+m}$.
\end{lemma}

The next lemma has been given in the proof of \cite[ Theorem 1]{4}.
\begin{lemma}\label{03}
Let $T\in {{\mathcal M}_{n}}$. If $f,g$ are non-negative continuous functions on $\left[ 0,\infty  \right)$ satisfying $f\left( t \right)g\left( t \right)=t$, ($t\ge 0$), then
\[\left[ \begin{matrix}
   {{f}^{2}}\left( \left| T \right| \right) & {{T}^{*}}  \\
   T & {{g}^{2}}\left( \left| {{T}^{*}} \right| \right)  \\
\end{matrix} \right]\ge O.\]
\end{lemma}

Now we present the following mixed CS inequality.

\begin{theorem}\label{02}
Let $A,B,C \in \mathcal M_{n}$ be such that $A,B\geq O$. If  $\left[ \begin{matrix}
   A & {{C}^{*}}  \\
   C & B  \\
\end{matrix} \right]\ge O$, then for any vectors $x,y\in\mathbb{C}^n,$ the following holds true
{\footnotesize	
\[{{\left| \left\langle Cx,y \right\rangle  \right|}^{2}}\le \left\langle \left( U\left[ \begin{matrix}
   A & O  \\
   O & O  \\
\end{matrix} \right]{{U}^{*}}+V\left[ \begin{matrix}
   O & O  \\
   O & B  \\
\end{matrix} \right]{{V}^{*}} \right)\left[ \begin{matrix}
   x  \\
   0  \\
\end{matrix} \right],\left[ \begin{matrix}
   x  \\
   0  \\
\end{matrix} \right] \right\rangle \left\langle \left( U\left[ \begin{matrix}
   A & O  \\
   O & O  \\
\end{matrix} \right]{{U}^{*}}+V\left[ \begin{matrix}
   O & O  \\
   O & B  \\
\end{matrix} \right]{{V}^{*}} \right)\left[ \begin{matrix}
   0  \\
   y  \\
\end{matrix} \right],\left[ \begin{matrix}
   0  \\
   y  \\
\end{matrix} \right] \right\rangle \]}
for some unitaries $U,V \in \mathcal M_{2n}$,
\end{theorem}
\begin{proof}
Let $x,y$ be any vectors. Assume that $\left[ \begin{matrix}
   A & {{C}^{*}}  \\
   C & B  \\
\end{matrix} \right]\ge O$. Utilizing Lemmas \ref{01} and \ref{04}, we have
{\small		\[\begin{aligned}
	& \left| \left\langle \left[ \begin{matrix}
	A & {{C}^{*}}  \\
	C & B  \\
	\end{matrix} \right]\left[ \begin{matrix}
	x  \\
	0  \\
	\end{matrix} \right],\left[ \begin{matrix}
	0  \\
	y  \\
	\end{matrix} \right] \right\rangle  \right|^2 \\ 
	& =\left| \left\langle \left( U\left[ \begin{matrix}
	A & O  \\
	O & O  \\
	\end{matrix} \right]{{U}^{*}}+V\left[ \begin{matrix}
	O & O  \\
	O & B  \\
	\end{matrix} \right]{{V}^{*}} \right)\left[ \begin{matrix}
	x  \\
	0  \\
	\end{matrix} \right],\left[ \begin{matrix}
	0  \\
	y  \\
	\end{matrix} \right] \right\rangle  \right| \\ 
	& \le \left\langle \left( U\left[ \begin{matrix}
	A & O  \\
	O & O  \\
	\end{matrix} \right]{{U}^{*}}+V\left[ \begin{matrix}
	O & O  \\
	O & B  \\
	\end{matrix} \right]{{V}^{*}} \right)\left[ \begin{matrix}
	x  \\
	0  \\
	\end{matrix} \right],\left[ \begin{matrix}
	x \\
	0  \\
	\end{matrix} \right] \right\rangle \left\langle \left( U\left[ \begin{matrix}
	A & O  \\
	O & O  \\
	\end{matrix} \right]{{U}^{*}}+V\left[ \begin{matrix}
	O & O  \\
	O & B  \\
	\end{matrix} \right]{{V}^{*}} \right)\left[ \begin{matrix}
	0  \\
	y  \\
	\end{matrix} \right],\left[ \begin{matrix}
	0 \\
	y  \\
	\end{matrix} \right] \right\rangle 	 
	\end{aligned}\]}
for some unitaries $U,V \in \mathcal M_{2n}$, which completes the proof.
\end{proof}


\begin{theorem}\label{12}
Let $T \in \mathcal M_{n}$ and let $x,y\in\mathbb{C}^n$  be any vectors. If $f,g$ are non-negative continuous functions on $\left[ 0,\infty  \right)$ satisfying $f\left( t \right)g\left( t \right)=t$, ($t\ge 0$), then
{\tiny		
\[\begin{aligned}
  & {{\left| \left\langle Tx,y \right\rangle  \right|}^{2}} \\ 
 & \le \left\langle \left( U\left[ \begin{matrix}
   {{f}^{2}}\left( \left| T \right| \right) & O  \\
   O & O  \\
\end{matrix} \right]{{U}^{*}}+V\left[ \begin{matrix}
   O & O  \\
   O & {{g}^{2}}\left( \left| {{T}^{*}} \right| \right)  \\
\end{matrix} \right]{{V}^{*}} \right)\left[ \begin{matrix}
   x  \\
   0  \\
\end{matrix} \right],\left[ \begin{matrix}
   x  \\
   0  \\
\end{matrix} \right] \right\rangle \left\langle \left( U\left[ \begin{matrix}
   {{f}^{2}}\left( \left| T \right| \right) & O  \\
   O & O  \\
\end{matrix} \right]{{U}^{*}}+V\left[ \begin{matrix}
   O & O  \\
   O & {{g}^{2}}\left( \left| {{T}^{*}} \right| \right)  \\
\end{matrix} \right]{{V}^{*}} \right)\left[ \begin{matrix}
   0  \\
   y  \\
\end{matrix} \right],\left[ \begin{matrix}
   0  \\
   y  \\
\end{matrix} \right] \right\rangle
\end{aligned}\]}
for some unitaries $U,V \in \mathcal M_{2n}$. In particular,
{\tiny		
\[\begin{aligned}
  & {{\left| \left\langle Tx,x \right\rangle  \right|}^{2}} \\ 
 & \le \left\langle \left( U\left[ \begin{matrix}
   {{f}^{2}}\left( \left| T \right| \right) & O  \\
   O & O  \\
\end{matrix} \right]{{U}^{*}}+V\left[ \begin{matrix}
   O & O  \\
   O & {{g}^{2}}\left( \left| {{T}^{*}} \right| \right)  \\
\end{matrix} \right]{{V}^{*}} \right)\left[ \begin{matrix}
   x  \\
   0  \\
\end{matrix} \right],\left[ \begin{matrix}
   x  \\
   0  \\
\end{matrix} \right] \right\rangle \left\langle \left( U\left[ \begin{matrix}
   {{f}^{2}}\left( \left| T \right| \right) & O  \\
   O & O  \\
\end{matrix} \right]{{U}^{*}}+V\left[ \begin{matrix}
   O & O  \\
   O & {{g}^{2}}\left( \left| {{T}^{*}} \right| \right)  \\
\end{matrix} \right]{{V}^{*}} \right)\left[ \begin{matrix}
   0  \\
   x  \\
\end{matrix} \right],\left[ \begin{matrix}
   0  \\
   x  \\
\end{matrix} \right] \right\rangle.
\end{aligned}\]
}
\end{theorem}
\begin{proof}
We infer the desired result by combining Theorem \ref{02} with Lemma \ref{03}.
\end{proof}

\begin{remark}\label{13}
It follows from \cite[Corollary 2.1]{01} and Lemma \ref{03} that
\[\left[ \begin{matrix}
   {{f}^{2}}\left( \left| T \right| \right) & {{T}^{*}}  \\
   T & {{g}^{2}}\left( \left| {{T}^{*}} \right| \right)  \\
\end{matrix} \right]=U\left[ \begin{matrix}
   \frac{{{f}^{2}}\left( \left| T \right| \right)+{{g}^{2}}\left( \left| {{T}^{*}} \right| \right)}{2}+\mathfrak RT & O  \\
   O & O  \\
\end{matrix} \right]{{U}^{*}}+V\left[ \begin{matrix}
   O & O  \\
   O & \frac{{{f}^{2}}\left( \left| T \right| \right)+{{g}^{2}}\left( \left| {{T}^{*}} \right| \right)}{2}-\mathfrak RT  \\
\end{matrix} \right]{{V}^{*}}\]
for some unitaries $U,V \in \mathcal M_{2n}$. In particular, if $T$ is self-adjoint, we have
\[\left[ \begin{matrix}
   {{f}^{2}}\left( \left| T \right| \right) & {{T}}  \\
   T & {{g}^{2}}\left( \left| {{T}} \right| \right)  \\
\end{matrix} \right]=U\left[ \begin{matrix}
   \frac{{{f}^{2}}\left( \left| T \right| \right)+{{g}^{2}}\left( \left| {{T}} \right| \right)}{2}+2T & O  \\
   O & O  \\
\end{matrix} \right]{{U}^{*}}+V\left[ \begin{matrix}
   O & O  \\
   O & \frac{{{f}^{2}}\left( \left| T \right| \right)+{{g}^{2}}\left( \left| {{T}} \right| \right)}{2}-2T  \\
\end{matrix} \right]{{V}^{*}}.\]
\end{remark}

We present the following interesting inequalities among $T$ and its real and imaginary parts. 

\begin{corollary}\label{nee1}
Let $T \in \mathcal M_{n}$. If $f,g$ are non-negative continuous functions on $\left[ 0,\infty  \right)$ satisfying $f\left( t \right)g\left( t \right)=t$, ($t\ge 0$), then
\[\left\| T \right\|\le \frac{1}{2}\left( \left\| \frac{{{f}^{2}}\left( \left| T \right| \right)+{{g}^{2}}\left( \left| {{T}^{*}} \right| \right)}{2}+\mathfrak RT \right\|+\left\| \frac{{{f}^{2}}\left( \left| T \right| \right)+{{g}^{2}}\left( \left| {{T}^{*}} \right| \right)}{2}-\mathfrak RT \right\| \right).\]
\end{corollary}
\begin{proof}
We know that if $M=\left[ \begin{matrix}
   A & C^{*}  \\
   {{C}} & B  \\
\end{matrix} \right]\ge O$, then 
$M\ge 2N$, where $N=\left[ \begin{matrix}
   O & C^{*}  \\
   {{C}} & O  \\
\end{matrix} \right]$. Now by taking into account Remark \ref{13}, we get
\[\begin{aligned}
   2\left\| T \right\|&=2\left\| \left[ \begin{matrix}
   O & {{T}^{*}}  \\
   T & O  \\
\end{matrix} \right] \right\| \\ 
 & \le \left\| \left[ \begin{matrix}
   {{f}^{2}}\left( \left| T \right| \right) & {{T}^{*}}  \\
   T & {{g}^{2}}\left( \left| {{T}^{*}} \right| \right)  \\
\end{matrix} \right] \right\| \\ 
 & = \left\| \left[ \begin{matrix}
   \frac{{{f}^{2}}\left( \left| T \right| \right)+{{g}^{2}}\left( \left| {{T}^{*}} \right| \right)}{2}+\mathfrak RT & O  \\
   O & O  \\
\end{matrix} \right] \right\|+\left\| \left[ \begin{matrix}
   O & O  \\
   O & \frac{{{f}^{2}}\left( \left| T \right| \right)+{{g}^{2}}\left( \left| {{T}^{*}} \right| \right)}{2}-\mathfrak RT  \\
\end{matrix} \right] \right\| \\ 
 & =\left\| \frac{{{f}^{2}}\left( \left| T \right| \right)+{{g}^{2}}\left( \left| {{T}^{*}} \right| \right)}{2}+\mathfrak RT \right\|+\left\| \frac{{{f}^{2}}\left( \left| T \right| \right)+{{g}^{2}}\left( \left| {{T}^{*}} \right| \right)}{2}-\mathfrak RT \right\|.  
\end{aligned}\]
This completes the proof.
\end{proof}

A special case of Corollary \ref{nee1} says that
\[\left\| T \right\|\le \frac{1}{4}\left( \left\| \left| T \right|+\left| {{T}^{*}} \right|+2\mathfrak RT \right\|+\left\| \left| T \right|+\left| {{T}^{*}} \right|-2\mathfrak RT \right\| \right).\]

In the following, we have a mixed CS-type inequality for the real and imaginary parts of $T\in\mathcal{M}_n$. The importance of this result can be seen, for example, in Remark \ref{rem_imre} below.
\begin{theorem}\label{14}
Let $T\in {{\mathcal M}_{n}}$. If $f,g$ are non-negative continuous functions on $\left[ 0,\infty  \right)$ satisfying $f\left( t \right)g\left( t \right)=t$, ($t\ge 0$), then for any vectors $x,y\in\mathbb{C}^n$,
\[\left| \left\langle \mathfrak RTx,y \right\rangle  \right|\le \frac{1}{2}\sqrt{\left\langle \left( {{f}^{2}}\left( \left| T \right| \right)+{{f}^{2}}\left( \left| {{T}^{*}} \right| \right) \right)x,x \right\rangle \left\langle \left( {{g}^{2}}\left( \left| T \right| \right)+{{g}^{2}}\left( \left| {{T}^{*}} \right| \right) \right)y,y \right\rangle },\]
and
\[\left| \left\langle \mathfrak ITx,y \right\rangle  \right|\le \frac{1}{2}\sqrt{\left\langle \left( {{f}^{2}}\left( \left| T \right| \right)+{{f}^{2}}\left( \left| {{T}^{*}} \right| \right) \right)x,x \right\rangle \left\langle \left( {{g}^{2}}\left( \left| T \right| \right)+{{g}^{2}}\left( \left| {{T}^{*}} \right| \right) \right)y,y \right\rangle }.\]
\end{theorem}
\begin{proof}
By Lemma \ref{03},
\[\left[ \begin{matrix}
   {{f}^{2}}\left( \left| A \right| \right) & {{A}^{*}}  \\
   A & {{g}^{2}}\left( \left| {{A}^{*}} \right| \right)  \\
\end{matrix} \right]\ge O\text{ and }\left[ \begin{matrix}
   {{f}^{2}}\left( \left| B \right| \right) & {{B}^{*}}  \\
   B & {{g}^{2}}\left( \left| {{B}^{*}} \right| \right)  \\
\end{matrix} \right]\ge O,\]
which imply that
\[\left[ \begin{matrix}
   {{f}^{2}}\left( \left| A \right| \right)+{{f}^{2}}\left( \left| B \right| \right) & {{A}^{*}}+{{B}^{*}}  \\
   A+B & {{g}^{2}}\left( \left| {{A}^{*}} \right| \right)+{{g}^{2}}\left( \left| {{B}^{*}} \right| \right)  \\
\end{matrix} \right]\ge O.\]
This is equivalent to saying that
\begin{equation}\label{07}
{{\left| \left\langle \left( A+B \right)x,y \right\rangle  \right|}^{2}}\le \left\langle \left( {{f}^{2}}\left( \left| A \right| \right)+{{f}^{2}}\left( \left| B \right| \right) \right)x,x \right\rangle \left\langle {{g}^{2}}\left( \left| {{A}^{*}} \right| \right)+{{g}^{2}}\left( \left| {{B}^{*}} \right| \right)y,y \right\rangle
\end{equation}
for any vectors $x,y$. Letting $A=T$ and $B={{T}^{*}}$, in \eqref{07}, we get the first inequality. The second inequality follows directly from the first inequality by setting $T=\textup i{{T}^{*}}$.
\end{proof}

\begin{remark}
The inequalities in Theorem \ref{14} are equivalent to
\begin{equation}\label{16}
\left[ \begin{matrix}
   \frac{{{f}^{2}}\left( \left| T \right| \right)+{{f}^{2}}\left( \left| {{T}^{*}} \right| \right)}{2} & \mathfrak RT  \\
   \mathfrak RT & \frac{{{g}^{2}}\left( \left| T \right| \right)+{{g}^{2}}\left( \left| {{T}^{*}} \right| \right)}{2}  \\
\end{matrix} \right]\ge O\text{ and }\left[ \begin{matrix}
   \frac{{{f}^{2}}\left( \left| T \right| \right)+{{f}^{2}}\left( \left| {{T}^{*}} \right| \right)}{2} & \mathfrak IT  \\
   \mathfrak IT & \frac{{{g}^{2}}\left( \left| T \right| \right)+{{g}^{2}}\left( \left| {{T}^{*}} \right| \right)}{2}  \\
\end{matrix} \right]\ge O.
\end{equation}
From the above discussion, we get
\begin{equation}\label{15}
\left[ \begin{matrix}
   {{f}^{2}}\left( \left| T \right| \right)+{{f}^{2}}\left( \left| {{T}^{*}} \right| \right) & \mathfrak RT+\mathfrak IT  \\
   \mathfrak RT+\mathfrak IT & {{g}^{2}}\left( \left| T \right| \right)+{{g}^{2}}\left( \left| {{T}^{*}} \right| \right)  \\
\end{matrix} \right]\ge O.
\end{equation}
\end{remark}

\begin{remark}\label{rem_imre}
Ando \cite[Theorem 3.4]{8} proved that if $C$ is self-adjoint and $\left[ \begin{matrix}
   A & C  \\
   C & B  \\
\end{matrix} \right]\ge O$, then $\pm C\le A\sharp B$.
Combining this with \eqref{15} implies that
\[\pm\left( \mathfrak RT+\mathfrak IT \right)\le \left( {{f}^{2}}\left( \left| T \right| \right)+{{f}^{2}}\left( \left| {{T}^{*}} \right| \right) \right)\sharp\left( {{g}^{2}}\left( \left| T \right| \right)+{{g}^{2}}\left( \left| {{T}^{*}} \right| \right) \right).\]
In particular,
\[\pm\left( \mathfrak RT+\mathfrak IT \right)\le \left| T \right|+\left| {{T}^{*}} \right|.\]
From \eqref{16}, we also infer that
\[\pm \mathfrak RT\le \left( \frac{{{f}^{2}}\left( \left| T \right| \right)+{{f}^{2}}\left( \left| {{T}^{*}} \right| \right)}{2} \right)\sharp \left( \frac{{{g}^{2}}\left( \left| T \right| \right)+{{g}^{2}}\left( \left| {{T}^{*}} \right| \right)}{2} \right),\]
and
\[\pm \mathfrak IT \le \left( \frac{{{f}^{2}}\left( \left| T \right| \right)+{{f}^{2}}\left( \left| {{T}^{*}} \right| \right)}{2} \right)\sharp \left( \frac{{{g}^{2}}\left( \left| T \right| \right)+{{g}^{2}}\left( \left| {{T}^{*}} \right| \right)}{2} \right).\]
Especially,
\begin{equation}\label{eq_reim_t}
\pm \mathfrak RT \le \frac{\left| T \right|+\left| {{T}^{*}} \right|}{2}\text{ and } \pm \mathfrak IT \le \frac{\left| T \right|+\left| {{T}^{*}} \right|}{2}.
\end{equation}

At this point, it might be thought that the last two inequalities imply 
$$|\mathfrak{R}T|\leq \frac{|T|+|T^*|}{2}\;{\text{or}}\;|\mathfrak{I}T|\leq \frac{|T|+|T^*|}{2}.$$ In fact, neither is true. Indeed, let $$T=\left[\begin{array}{ccc}0&1&0\\0&0&1\\0&0&0\end{array}\right].$$
Direct calculations show that 
$$|T|+|T^*|=\left(\begin{array}{ccc}1&0&0\\0&2&0\\0&0&1\end{array}\right)\;{\text{and}}\;T+T^*=\left(\begin{array}{ccc}0&1&0\\1&0&1\\0&1&0\end{array}\right).$$ The singular values of $|T|+|T^*|$ are 2,1,1, while they are $\sqrt{2},\sqrt{2},0$ for $T+T^*.$ This shows that $|T+T^*|\not\leq|T|+|T^*|$. That is $|\mathfrak{R}T|\not\leq\frac{|T|+|T^*|}{2}.$
\end{remark}

	\section*{Acknowledgment} The authors would like to express their sincere thanks to Professor J. C. Bourin for bringing their attention to a serious mistake in the first draft of this work. In particular, he provided us with the example in Remark  \ref{rem_imre} to confute the claim $|\mathfrak{R}T|\leq \frac{|T|+|T^*|}{2}.$

\section*{Declarations}

\subsection*{Ethical approval}
This declaration is not applicable.
\subsection*{Competing interest}
The authors declare that they have no competing interest.
\subsection*{Authors' contribution}
The authors have contributed equally to this work.
\subsection*{Funding}
The authors did not receive any funding to accomplish this work.
\subsection*{Availability of data and materials}
This declaration is not applicable

\vskip 0.3 true cm

\noindent{\tiny (M. Sababheh) Department of basic sciences, Princess Sumaya University for Technology, Amman, Jordan}
	
\noindent	{\tiny\textit{E-mail address:} sababheh@psut.edu.jo; sababheh@yahoo.com}

\vskip 0.3 true cm 	

\noindent{\tiny (C. Conde)  Instituto de Ciencias, Universidad Nacional de General Sarmiento  and  Consejo Nacional de Investigaciones Cient\'ificas y Tecnicas, Argentina}

\noindent{\tiny \textit{E-mail address:} cconde@campus.ungs.edu.ar}

\vskip 0.3 true cm

\noindent{\tiny (H. R. Moradi) Department of Mathematics, Payame Noor University (PNU), P.O. Box, 19395-4697, Tehran, Iran
	
\noindent	\textit{E-mail address:} hrmoradi@mshdiau.ac.ir}
\end{document}